\documentclass[10pt]{article}

\usepackage{bigints}

\usepackage{amssymb}
\usepackage{mathrsfs}
\usepackage{amsmath}
\usepackage{amsfonts}

\usepackage{amsthm}
\newtheorem{theo}{Theorem}[section]

\newtheorem{coro}{Corollary}[section]
\newtheorem{lemm}{Lemma}[section]

\newtheorem{main}{Main Theorem}

\theoremstyle{definition}
\newtheorem{defi}{Definition}[section]

\theoremstyle{remark}

\newtheorem{rema}{Remark}[section]

\usepackage{hyperref}

\begin{document}

\title{\large{\textbf{Ricci flow with bounded curvature integrals}}}

\author{Shota Hamanaka\thanks{supported in doctoral program in Chuo University, Japan.}}

\date{}

\maketitle

\begin{abstract}
In this paper, we study the Ricci flow on a closed manifold and finite time interval $[0,T)~(T < \infty)$
on which certain integral curvature energies are finite.
We prove that in dimension four, such flow converges to a smooth Riemannian manifold
except for finitely many orbifold singularities.
We also show that in higher dimensions, the same assertions hold for a closed Ricci flow satisfying another conditions of integral curvature bounds.
Moreover, we show that such flows can be extended over $T$
by an orbifold Ricci flow.
\end{abstract}

\section{Introduction}
~~In this paper, we will consider extension problem for the Ricci flow on a closed manifold.
Given a $n$-manifold $M~(n \ge 2),$ a family of smooth Riemannian metrics $g(t)$ on $M$ is called a Ricci flow
(introduced by Hamilton in \cite{hamilton1982three}) 
on the time interval $[0,T) \subset \mathbb{R}$ 
if it satisfies 
\[
\frac{\partial g(t)}{\partial t} = - 2\,\mathrm{Ric}_{g(t)},~~\mathrm{for~all}~t \in [0,T),
\]
where $\mathrm{Ric}_{g(t)} = (R_{ij})$ denotes the Ricci curvature tensor of $g(t).$
One hopes that the Ricci flow will deform any Riemannian metric to some canonical metric, such as Einstein metric.
Hamilton \cite{hamilton1982three} first proved that for any smooth initial metric, the flow will exist uniquely for a short time 
and his proof was simplified by DeTurck \cite{deturck1983deforming}.
The next immediate question is the maximal existence time for the Ricci flow with respect to the initial metric.
In \cite{hamilton2formation}, Hamilton proved that
a Ricci flow on a closed manifold develops a singularity at a finite time $T$
(i.e., $T$ is the maximal existence time of the flow)
if and only if
the maximum of the norm $|\mathrm{Rm}| = \left( R_{ijkl} \cdot R^{ijkl} \right)^{1/2}$ of the Riemannian curvature tensor $\mathrm{Rm} = \mathrm{Rm}_{g(t)} = (R_{ijkl})$ blows up at $T.$
{\v{S}}e{\v{s}}um \cite{vsevsum2005curvature} proved that indeed a bound on the Ricci curvature rather than the full Riemannian curvature tensor sufficies to extend the closed Ricci flow.
A different approach was adopted by Wang \cite{wang2008conditions}
and consists of considering integral bounds rather that point-wise ones.
In this paper, we also consider a closed Ricci flow with some integral bounds.
Here, we state our two main results.
First one is the following.
\begin{main}[cf.~{\cite[Corollary~1.11]{bamler2017heat}}]
\label{maintheo}
Let $(M^{4}, g(t))_{t \in [0,T)}~(T < \infty)$ be a 4-dimensional closed (i.e., $M$ is smooth compact and connected manifold without boundary) Ricci flow satisfying
\[
(*)~~~\big| \big| \sup_{M} |R_{g(t)}| \big| \big|_{L^{1}([0, T))} \le C < +\infty
\]
for some positive constant $C,$ where $R_{g(t)}$ denotes the scalar curvature of $g(t).$
Then there exists a positive constsnat $\varepsilon = \varepsilon(M, g(0), T)$ such that the following holds :
~$(*)_{p_{0}, \varepsilon}$~~For fixed $p_{0} > 2,$ assume that there exists $r > 0$ such that 
\[
\sup_{t \in [0,T)}~|| R_{g(t)} ||_{L^{p_{0}}(B(x, r, t))} \le \varepsilon
\]
for all $x \in M,$ where $B(x,r,t)$ denotes the geodesic open ball centered at $x$ of radius $r$ with respect to $g(t).$
Then $(M, g(t))$ converges to an orbifold in the smooth Cheeger-Gromov sense.
More specifically, we can find a decomposition $M = M^{\mathrm{reg}} \bigcup M^{\mathrm{sing}}$
with the following properties:

\noindent
(1)~$M^{\mathrm{reg}}$ is open and connected in $M$,

\noindent
(2)~$M^{\mathrm{sing}}$ is a zero set with respect to the Riemannian volume measure
$dvol_{g(t)}$ for all $t \in [0,T),$

\noindent
(3)~$g(t)$ smoothly converges to a Riemannian metric $g_{T}$ on $M^{\mathrm{reg}}$ as $t \rightarrow T.$

\noindent
(4)~$(M^{\mathrm{reg}}, g_{T})$ can be compactified to a metric space $(\Bar{M}^{\mathrm{reg}}, \bar{d})$
by adding finitely many points and the differentiable structure on $M^{\mathrm{reg}}$ can be extended to a smooth orbifold structure on $\Bar{M}^{\mathrm{reg}}$ 
such that the orbifold singularities are of cone type,

\noindent
(5)~Around every orbifold singularity of $(\Bar{M}^{\mathrm{reg}}, \bar{d})$ the metric $g_{T}$
satisfies 
\[
| \nabla^{m} \mathrm{Rm} | < o (\rho^{-2 -m})~\mathrm{and}~| \nabla^{m} \mathrm{Ric} | < O (\rho^{-1 -m -\frac{2}{p_{0}}})
~\mathrm{as}~\rho \rightarrow 0,
\]
where $\rho$ denotes the distance to the singularity.
Furtheremore, for every $\varepsilon > 0$ we can find a smooth orbifold metric $\bar{g}_{\varepsilon}$ on $\Bar{M}^{\mathrm{reg}}$ such that the following holds:
\[
|| g_{T} - \bar{g}_{\varepsilon} ||_{C^{0} (M^{\mathrm{reg}}, \bar{g}_{\varepsilon})}
+ || g_{T} - \bar{g}_{\varepsilon} ||_{W^{2,2} (M^{\mathrm{reg}}, \bar{g}_{\varepsilon})} < \varepsilon.
\]
Here, the $C^{0}$ and $W^{2,2}$-norms are taken with respect to $\bar{g}_{\varepsilon}.$
\end{main}
\begin{rema}
(1)~In Main Theorem \ref{maintheo}, the metric space as limit of the flow is a singular space in the sense of {\cite[Definition~1.16]{bamler2020structure}} (it follows that the space is also length space by the same argument in {\cite[Proof~of~Theorem~6.7]{simon2015extending}}. See also {\cite[Theorem~2.4.16]{burago2001course}}).

\noindent
(2)~Bamler and Zhang \cite{bamler2017heat} proved that the similar result in Main Theorem \ref{maintheo}
under the uniformly bounded scalar curvature assumption (see {\cite[Corollary~1.11]{bamler2017heat}} for more detail).
Later, Bamler \cite{bamler2018convergence} generalized this result to higher dimensions.
The assumption $(b)$ in Main Theorem \ref{maintheo} admits the possibility that the scalar curvature
blows up at most $(T -t)^{- \alpha}~(0 < \alpha < 1)$-rate.
So the assumption $(b)$ may be weaker than one in \cite{bamler2017heat}.
Note also that, as  far as we know, there are no examples that satisfy the assumptions and
have some singularities as in Main Theorem \ref{maintheo}, {\cite[Corollary~1.11]{bamler2017heat}} or {\cite[Theorem~1.1]{bamler2018convergence}}.
On the other hand, at least in dimension $4 \le n \le 7,$
it has been conjectured that a bound on the scalar curvature could potentially be sufficient to extend the flow
(see \cite{buzano2020local}).

\end{rema}
Second one is the following.
\begin{main}[cf.~{\cite[Theorem~A]{ye2008curvature}}]
\label{maintheo2}
Let $(M^{n}, g(t))_{t \in [0,T)}~(T < \infty)$ be a $n$-dimensional ($n \ge 5$) closed Ricci flow satisfying $(*)$
in Main Theorem \ref{maintheo}.
Then there exists a positive constsnat $\varepsilon = \varepsilon(M, g(0), n, T)$ such that the following holds :
Suppose that 
\[
\sup_{t \in [0,T)}~|| \mathrm{Rm}_{g(t)} ||_{L^{n/2}(M)} < +\infty
\]
and $(*)_{p_{0}, \varepsilon}$ for some $p_{0} > n/2$ in Main Theorem \ref{maintheo} holds.
Then the assertions (1)-(5) as in Main Theorem \ref{maintheo} hold.
\end{main}

\begin{rema}
In {\cite[Theorem~A]{ye2008curvature}}, Ye proved an extension theorem under a smallness assumption on the $(n/2, +\infty)$-mixed norm. But we know nothing about the relationship between the constant $\delta_{0}$ in {\cite[Theorem~A]{ye2008curvature}} and the constant $\varepsilon$ in Main Theorem \ref{maintheo2}.
See also related work by Chen and Wang \cite{chen2012space}.
\end{rema}

In the rest of this section, we shall explain backgrounds.
As stated above, Wang \cite{wang2008conditions} considered integral bounds rather that point-wise ones.
First one consists of integral bound of the Riemannian curvature tensor:
\begin{theo}[{\cite[Theorem~1.1]{wang2008conditions}}]
\label{wang1}
Let $(M^{n}, g(t))_{t \in [0,T)}~(T < + \infty)$ be a closed Ricci flow.
If
\[
||\mathrm{Rm} (\cdot, t)||_{L^{\alpha} (M \times [0,T))} < + \infty,~~\alpha \ge \frac{n+2}{2},
\]
then this flow can be extended smoothly over $T.$
\end{theo}
Second one consists of the integral bound of sclar curvature and uniformly Ricci curvature bound from below:
\begin{theo}[{\cite[Theorem~1.2]{wang2008conditions}}]
\label{wang2}
Let $(M^{n}, g(t))_{t \in [0,T)}~(T < +\infty)$ be a closed Ricci flow.
If
\[
(1)~\mathrm{Ric} (x,t) \ge -A \cdot g(x,t),~~\mathrm{for~all}~(x,t) \in M \times [0,T),~\mathrm{for~some}~A \in \mathbb{R},
\]
\[
(2)~|| R (\cdot, t) ||_{L^{\alpha} (M \times [0,T))} < + \infty,~~\alpha \ge \frac{n+2}{2},
\]
then this flow can be extended smoothly over $T.$
\end{theo}
\begin{rema}
In Theorem \ref{wang1}, \ref{wang2}, let $\alpha = + \infty,$
one can recover Hamilton's \cite{hamilton2formation} and {\v{S}}e{\v{s}}um's \cite{vsevsum2005curvature} results.
\end{rema}
Recently, Di Matteo \cite{matteo2020mixed} generalized these results using mixed integral norms.
\begin{theo}[{\cite[Theorem~1.2]{matteo2020mixed}}]
\label{matteo1}
Let $(M^{n}, g(t))_{t \in [0,T)}~(T < +\infty)$ be a smooth Ricci flow
such that $(M, g(t))$ is complete and has bounded curvature for every $t$ in $[0,T).$
Assume that the initial slice $(M, g(0))$ satisfies $\mathrm{inj} (M, g(0)) > 0.$
Assume also that
\[
\big| \big| || \mathrm{Rm} (\cdot, t) ||_{L^{\alpha}(M)} \big| \big|_{L^{\beta}([0,T))} < +\infty
\]
for some pair $(\alpha, \beta) \in (1, +\infty) \subset \mathbb{R}$ with
\[
\alpha \ge \frac{n}{2} \frac{\beta}{\beta-1}.
\]
Then this flow can be extended smoothly over $T.$
\end{theo}

\begin{theo}[{\cite[Theorem~1.3]{matteo2020mixed}}]
\label{matteo2}
Let $(M^{n}, g(t))_{t \in [0,T)}~(T < +\infty)$ be a smooth Ricci flow
such that $(M, g(t))$ is complete and has bounded curvature for every $t$ in $[0,T).$
Assume that the initial slice $(M, g(0))$ satisfies $\mathrm{inj} (M, g(0)) > 0.$
Assume also that the following conditions hold:
\[
(1)~\mathrm{Ric}(x,t) \ge -A \cdot g(x,t)~~\mathrm{for~all}~(x,t) \in M \times [0,T),
\]
\[
(2)~\big| \big| || R (\cdot, t) ||_{L^{\alpha}(M)} \big| \big|_{L^{\beta}([0,T))} < +\infty
\]
for some $A \in \mathbb{R}$ and pair $(\alpha, \beta) \in (1, +\infty) \subset \mathbb{R}$ with
\[
\alpha \ge \frac{n}{2} \frac{\beta}{\beta - 1}.
\]
Then this flow can be extended smoothly over $T.$
\end{theo}
In the present paper, we will consider a closed $n$-dimensional Ricci flow under the conditions 
correspoinding to $(\alpha, \beta)=$``$(+\infty, 1)$'' and ``$(n/2, +\infty)$''
in the above Theorem \ref{matteo1} or \ref{matteo2}.
These conditions are listed as follows :
Let $(M^{n}, g(t))_{t \in [0,T)}~(T < +\infty)$ be a closed $n$-dimensional Ricci flow
satisfying the following $(A)-(B)$ or (a)-(c):

\[
(A)~\sup_{t \in [ 0, T )} \big| \big| |\mathrm{Rm}|_{g(t)} \big| \big|^{n/2}_{L^{n/2}(M)} \le A_{1},
\]
\[
(B)~\big| \big| \sup_{M} |\mathrm{Rm}|_{g(t)} \big| \big|_{L^{1}([0, T))} \le A_{2}
\]
for some positive constants $A_{1}, A_{2},$ or
\[
(a)~\sup_{t \in [ 0, T )} || R_{g(t)} ||^{n/2}_{L^{n/2}(M)} \le C_{1}
\]
for some positive constant $C_{1},$
\[
(b)~\big| \big| \sup_{M} |R_{g(t)}| \big| \big|_{L^{1}([0, T))} \le C_{2}
\]
for some positive constant $C_{2},$
and
\[
(c)~\mathrm{Ric} (g(t)) \ge - C_{3} \cdot g(t)
\]
for some constant $C_{3}.$

\begin{rema}
\label{rema1}
(1)~Under the assumption (c), the condition (b) implies the following condition $(b')$ :
\[
(b')~\big| \big| \sup_{M} |\textrm{Ric}_{g(t)}| \big| \big|_{L^{1}([0, T))} \le C_{4} := n (C_{2} + (n-1)C_{3}).
\]

\noindent
(2)~As in {\cite[Example]{wang2008conditions}},
under the condition (a) only,
we can not expect that such flow can be extended smoothly over $T.$
The same example also shows that
we can not expect that Ricci flow satisfying (a) and 
\[
\big| \big| \log ( \sup_{M} |R_{g(t)}|) \big| \big|_{L^{1}([0, T))} < +\infty
\]
can be extended smoothly over $T.$
\end{rema}
As mentioned above, 
$T < +\infty$ is the maximal exsitence time if and only if
\[
\limsup_{t \rightarrow T} |\mathrm{Rm}_{t}|_{t} = +\infty.
\] 
Moreover, this condition is equivalent to 
\[
\sup_{M} |\mathrm{Rm}_{t}|_{t} \ge \frac{1}{8(T -t)}~~\forall t \in [0,T) 
\]
from the evolution of $|\mathrm{Rm}|^{2}$ under the Ricci flow and Hamilton's long-existence theorem
(see {\cite[Corollary~7.7]{chow2008ricci}}).
This condition implies that
\[
\big| \big| \sup_{M} |\mathrm{Rm}|_{g(t)} \big| \big|_{L^{1}([0,T))} = +\infty.
\]
Therefore, if $(M^{n}, g(t))_{t \in [0,T)}~(T < +\infty)$
satisfies $(B),$ then the flow can be smoothly extended over $T.$
Hence we will consider the following weaker conditions $(A), (B')$ :
\[
(A)~\sup_{t \in [ 0, T )} \big| \big| |\mathrm{Rm}|_{g(t)} \big| \big|^{n/2}_{L^{n/2}(M)} \le A_{1},
\]
\[
(B')~\big| \big| \sup_{M} |\mathrm{Ric}|_{g(t)} \big| \big|_{L^{1}([0, T))} \le A_{2}.
\]
However, under the assumption $(B')$ (or equivalently $(b')$), the flow can be extended smoothly over $T$
by the following result (He \cite{he2014remarks}) :
\begin{theo}[{\cite[Theorem~1.1]{he2014remarks}}]
If a closed Ricci flow $(M, g(t))_{t \in [0,T)}$ satisfies
\[
\big| \big| \sup_{M} |\mathrm{Ric}|_{g(t)} \big| \big|_{L^{1}([0, T))} < +\infty,
\]
then the flow can be extended smoothly over $T.$
\end{theo}

\noindent
Hence we will consider the assumptions $(A), (b)$ or $(a), (b)$. 
Main Theorem \ref{maintheo} relates to $(a), (b)$ and Main Theorem \ref{maintheo2} relates to $(A), (b).$

In dimension two, 
there is no difference between scalar curvature and Ricci curvature.
Hence, from the above observation, 
closed two-dimensional Ricci flow satisfying (a) and (b) can be smoothly extended over $T.$
In dimension three,
we have Hamilton-Ivey pinching result({\cite[Theorem~10.2.1]{topping2006lectures}}).
Hence, from the above observation, 
closed three-dimensional Ricci flow satisfying (a) and (b) can be smoothly extended over $T.$
Therefore we will consider in dimension $\ge 4.$
Some results about closed Ricci flow on which the scalar curvature is uniformly bounded in space and time
are already known (see  \cite{simon2015extending}, \cite{bamler2017heat}, \cite{bamler2018convergence}, \cite{bamler2019heat}, \cite{buzano2020local}).
We will borrow ideas in \cite{bamler2017heat} in some respects of the proof of Main Theorems.

This paper is organized as follows.
In section 2, we firstly prepare some basic properties and secondly show necessary lemmas.
Finally, we will prove Main Theorem \ref{maintheo} by using these lemmas.
In section 3, we will prove Main Theorem \ref{maintheo2} by the same method used in the proof of Main Theorem \ref{maintheo}.
In section 4, we will prove that the flow considered in section 2, 3 can be extended by an orbifold Ricci flow.
Finally, in section 5, we will describe a well-known lemma which is necessary for the proof of Main Theorems.

\section{Under scalar curvature integral bounds}
We begin to discuss some basic estimates.
We will write the Riemannian volume measure of $g(t)$ by $vol_{g(t)}$ or $vol_{t}$ briefly.
\begin{lemm}[Volume bounds]
\label{volume}
Assume (b).
Then there exist positive constants $V_{1} = V_{1}(M, g(0), C_{2}),~V_{2} = V_{2}(M, g(0), T)$ such that the following holds
\[
0 < V_{1} \le \mathrm{Vol} (M, g(t)) \le V_{2}
\]
for all $t \in [0,T),$ where $\mathrm{Vol} (M, g(t))$ denotes the volume of $(M, g(t)).$
Here, $C_{2}$ denotes the upper bound in $(b)$ stated in section 1. 
\end{lemm}

\begin{proof}
The proof is almost same as in {\cite[p.6~(3.8)]{simon2015extending}}.
To be self-contained, we will describe the proof without omitting.

We firstly prove the volume estimate from below.
Under the Ricci flow (see \cite{hamilton1982three}, \cite{simon2015extending}),
\[
\frac{d}{dt} \mathrm{Vol} (M, g_{t}) = - \int_{M} R_{g_{t}} dvol_{g_{t}}
\ge - (\sup_{M} |R_{g_{t}}|) \cdot \mathrm{Vol} (M, g_{t}) 
\]
Integrating this and from the assumption (b), we obtain
\[
\log \frac{\mathrm{Vol} (M, g_{t})}{\mathrm{Vol} (M, g_{0})} \ge - C_{2}.
\]
Hence we obtain that
\[
\mathrm{Vol} (M, g_{t}) \ge e^{- C_{2}} \mathrm{Vol} (M, g_{0}) > 0
\]
for all $t \in [ 0, T ).$

Next, we will prove the estimate from above.
By the evolution of the scalar curvature and maximum principle (see {\cite[Lemma~3.36]{bamler2020structure}}),
we obtain that 
\[
\begin{split}
\frac{d}{dt} \mathrm{Vol} (M, g_{t}) &= - \int_{M} R_{g_{t}} dvol_{g_{t}}
\le - (\inf_{M} R_{g_{t}}) \cdot \mathrm{Vol} (M, g_{t}) \\
\le \frac{n}{2t} \mathrm{Vol} (M, g_{t})
\end{split}
\]
Hence we obtain that 
\[
\mathrm{Vol} (M, g_{t}) \le e^{n/2}t \cdot \mathrm{Vol} (M, g_{0}) < e^{n/2}T \cdot \mathrm{Vol} (M, g_{0}).
\]
\end{proof}

\begin{lemm}[Non-inflating estimate]
\label{non-inflating}
Assume (a) and (b).
Then there exists a positive constant $\kappa_{1} = \kappa_{1}(M, g(0), n, T, C_{2}) > 0$ such that 
\[
\int_{B(x,t,r)}~dvol_{g(t)} =: \mathrm{Vol}_{t}~(B(x, t, r)) \le \kappa_{1} r^{n}
\]
for all $r \in (0, 1].$
Here, $C_{2}$ denotes the upper bound in $(b)$ stated in section 1.
\end{lemm}

\begin{proof}
The proof of Zhang {\cite[Theorem 1.1 (a)]{zhang2011bounds}} in the uniformly bounded scalar curvature case
can be adapted to this case with some minor modifications.
We only indicate the necessary changes. 
Step 1 and 2 in it also hold in this situation 
(Note that they used the logarithmic Sobolev inequality under the Ricci flow in Step 2,
but since we can take $\varepsilon$ arbitrarily at that point, 
we do not need the uniform scalar curvature bound).
In Step 3,
we have to modify the arguments in the following 
(we will use the same notation as in \cite{zhang2011bounds}):

for $x \in M$ such that $d_{t_{0}} (x_{0}, x) \le r$ where $d_{t_{0}}$ denotes the natural distance function with respect to $g(t_{0}),$ 
\[
\begin{split}
G(x_{0}, t_{0} - r^{2} ; x, t_{0}) 
&\ge \frac{c_{1} J(t_{0})}{r^{n}} e^{- 2 c_{2}} 
e^{- \frac{1}{r} \int^{t_{0}}_{t_{0} - r^{2}} \sqrt{t_{0} - s} R(x, s) ds} \\
&\ge \frac{c_{1} J(t_{0})}{r^{n}} e^{- 2 c_{2}} 
e^{- \frac{1}{r} \sqrt{t_{0} - (t_{0} - r^{2})} \int^{t_{0}}_{t_{0} - r^{2}}  \sup_{M} |R(\bullet , s)| ds} \\
&\ge \frac{c_{1} J(t_{0})}{r^{n}} e^{- 2 c_{2}} 
e^{- C_{2}}.
\end{split}
\]
Here $G(y, s ; x,t)~(x, y \in M,~0 < s < t \le t_{0})$ denotes the conjugate heat kernel with respect to $(y,s) \in M \times (0,t)$ centered at $(x,t).$ 
Note that $G(y,s ; x,t) = K(x,t ; y,s),$ where $K(\cdot, \cdot ; y,s)$ denotes the heat kernel coupled with Ricci flow centered at $(y,s)$ as in the proof of Lemma \ref{distance2} below(see also \cite{bamler2020structure}, \cite{bamler2017heat}).
Hence we can continue the same arguments in the Zhang's proof and obtain that
\[
\mathrm{Vol}_{t} (B(x_{0}, t_{0}, r)) \le \kappa_{1} r^{n}
\]
for all $r \in (0, \sqrt{t_{0}}~].$
And we consider the parabollic rescaling flow $(\sqrt{t_{0}})^{-1} g (\sqrt{t_{0}} \cdot),$
then we can obtain the desired result since the volume growth estimate is invariant under the rescaling.
\end{proof}

\begin{lemm}[Diameter bounds]
Assume (a) and (b).
Then there exist positive constants $D_{1} = D_{1}(M, g(0), n, T, C_{2})$
and $D_{2} = D_{2}(M, g(0), n, T, C_{1}, C_{2})$ such that 
\[
0 < D_{1} \le \mathrm{diam} (M, g(t)) \le D_{2} < +\infty
\]
for all $t \in [0, T),$ where $\mathrm{diam} (M, g(t))$ denotes the diameter of $(M, g(t)).$
Here, $C_{1}, C_{2}$ denote respectively the upper bounds in $(a), (b)$ stated in section 1.
\end{lemm}

\begin{proof}
The uniform lower bound follows in the same way as in the proof of {\cite[Lemma~3.5]{simon2015extending}}
because we have volume lower bound and the non-inflating estimate
(Note that these follows from the assumption (b) as above).
Here, we do not have the non-collapsing estimate,
but we can get the upper diameter bound from {\cite[Theorem 2.4]{topping2005diameter}}.
In fact,
from H{\"o}lder's inequality, the volume bound from above and the assumtion (a), we obtain
\[
\int_{M} R^{\frac{n-1}{2}} dvol_{g_{t}} \le || R ||^{\frac{n-1}{2}}_{L^{n/2}} \cdot \bigl( \mathrm{Vol} (M, g(t)) \bigr)^{1/n} \le 
C_{1}^{\frac{n-1}{n}} \cdot V_{2}^{1/n} < +\infty.
\]
Hence, using Theorem 2.4 in \cite{topping2005diameter}, we obtain a diameter bound from above. 
\end{proof}
Next, we will show that the non-collapsing estimate.
To do this, we firstly prepare the following Sobolev inequality under the closed Ricci flow :
\begin{lemm}[The Sobolev inequaltiy under the Ricci flow({\cite[Theorem~$\textrm{D}^{*}$]{ye2007logarithmic}})]
\label{sobolev}
Let $(M^{n}, g(t))_{t \in [0,T}~(T < +\infty,~n \ge 3)$ be a closed $n$-dimensional Ricci flow.
Then there are some positive constants $A$ and $B$ depending only on the dimension $n$, $\min \{ \inf R_{g(0)},~0 \},$
a positive lower bound for $\mathrm{Vol} (M, g_{0}),$ an upper bound for the Sobolev constant of $(M, g_{0})$
, and $T,$
such that for each $t \in [0, T)$
and all $u \in W^{1,2}(M)$
there holds
\[
\left( \int_{M} |u|^{\frac{2n}{n-2}} dvol \right)^{\frac{n-2}{n}}
\le A \int_{M} \left( |\nabla u|^{2} + \frac{R}{4} u^{2} \right) dvol
+ B \int_{M} u^{2} dvol,
\]
where all geometric quantities, except $A$ and $B,$ are associated with $g(t).$
\end{lemm}
Thus, we can show that closed Ricci flow with small $L^{n/2}$-norm of the scalar curvature 
is local non-collapsed.
\begin{lemm}
\label{non-collapsing}
Let $(M^{n}, g)$ be a Riemannian $n$-manifold $(n \ge 3)$ which satisfies the Sobolev inequality as in the previous lemma.
Then there is a positive constant $\varepsilon = \varepsilon (A) > 0$ such that the following holds:
for $x \in M$ and $r > 0,$ if $|| R_{g} ||^{n/2}_{L^{n/2}( B(x, r) )} \le \varepsilon,$
then
\[
\mathrm{Vol}_{g} (B(x, l)) \ge \delta l^{n},~~\forall l \in (0, \min \{ 1, r \}]
\]
for some $\delta = \delta(A, B) > 0.$
Here $A, B$ denote the corresponding constants in Lemma \ref{sobolev}.
\end{lemm}

\begin{proof}
Since $|| R_{g} ||^{n/2}_{L^{n/2}( B(x, l) )} \le \varepsilon$ holds for all $0< l \le 1$ if $r > 1,$
it is sufficient to show the assertion in the case that $r \le 1.$

Fix a constant $a \in (0, 1).$
Take and fix $0 < \varepsilon < ( \frac{4a}{A} )^{n/2}.$
We suppose that the following holds
\[
\mathrm{Vol}_{g} (B(x, r)) < \delta r^{n},
\]
where 
\[
\delta := \min \left\{ \left( \frac{a - \left( \frac{A}{4}  \right) \varepsilon^{2/n}}{B} \right)^{n/2},~
\left( \frac{1 - a}{2^{n + 2} A} \right)^{n/2} \right\}.
\]
We will lead to a contradiction.
Set $\bar{g} := \frac{1}{r^{2}} g.$
Then we have for $\bar{g},$
\[
(*)~~~\mathrm{Vol}_{\bar{g}} (B(x, 1)) < \delta
\]
and
\[
|| R_{\bar{g}} ||^{n/2}_{L^{n/2}( B(x, 1) )} \le \varepsilon.
\]
In the rest of the proof, we will denote all geometric quantities(the geodesic balls, the norms, $dvol,$ and $R$) with respect to $\bar{g}.$
By the Sobolev inequality for $g$ and the assumption $r \le 1,$
for $u \in C^{\infty}(M)$ with support contained in $B(x, 1),$ 
we then have
\[
\begin{split}
\left( \int_{B(x, 1)} |u|^{\frac{2n}{n-2}} dvol \right)^{\frac{n-2}{n}}
&\le A \int_{B(x, 1)} \left(|\nabla u|^{2} + \frac{R}{4} u^{2} \right) dvol
+ B r^{2} \int_{B(x, 1)} u^{2} dvol \\
&\le A \int_{B(x, 1)} \left(|\nabla u|^{2} + \frac{R}{4} u^{2} \right) dvol
+ B \int_{B(x, 1)} u^{2} dvol.
\end{split}
\]
From H{\"o}lder's inequality and upper volume estimate, we have
\[
\int_{B(x, 1)} u^{2} dvol \le \delta^{2/n} \left( \int_{B(x, 1)} |u|^{\frac{2n}{n-2}} dvol \right)^{\frac{n-2}{n}}.
\]
And, from H{\"o}lder's inequality and the assumption (Note that the $L^{n/2}$-assumption is scale invariant), 
we have
\[
\int_{B(x, 1)} \frac{R}{4} u^{2} dvol \le \frac{\varepsilon^{2/n}}{4} \left( \int_{B(x, 1)} |u|^{\frac{2n}{n-2}} dvol \right)^{\frac{n-2}{n}}.
\]
Hence we deduce
\[
\begin{split}
\left( \int_{B(x, 1)} |u|^{\frac{2n}{n-2}} dvol \right)^{\frac{n-2}{n}}
&\le A \int_{B(x, 1)} |\nabla u|^{2} dvol \\
&~~~~~~~~~~+ \left( \frac{A \varepsilon^{2/n}}{4} + B \delta^{2/n} \right) \left( \int_{B(x, 1)} |u|^{\frac{2n}{n-2}} dvol \right)^{\frac{n-2}{n}} \\
&\le A \int_{B(x, 1)} |\nabla u|^{2} dvol + a \left( \int_{B(x, 1)} |u|^{\frac{2n}{n-2}} dvol \right)^{\frac{n-2}{n}}.
\end{split}
\]
It follows that
\[
\left( \int_{B(x, 1)} |u|^{\frac{2n}{n-2}} dvol \right)^{\frac{n-2}{n}}
\le \frac{A}{1 - a} \int_{B(x, 1)} |\nabla u|^{2} dvol.
\]
Then, we have the following volume bound as in {\cite[Proof~of~Proposition~2.1]{akutagawa1994yamabe}}
or \cite{carron1994sobolev}
(cf. {\cite[Theorem~$\textrm{E}^{*}$]{ye2007logarithmic}}).
\[
\mathrm{Vol} (B(x, \rho)) \ge \left( \frac{1 - a}{2^{n + 2} A} \right)^{n/2} \rho^{n}
\]
for all $B(x, \rho) \subset B(x, 1).$
Consequently we have 
\[
\mathrm{Vol} (B(x, 1)) \ge \left( \frac{1 - a}{2^{n + 2} A} \right)^{n/2}.
\]
This contradict the above upper volume bound $(*)$.
\end{proof}

\noindent
Thus, we will consider the following assumption $(a')$ which is stronger than (a):
\[
(a')~~~~~\mathrm{For~some}~r > 0,~\sup_{t \in [0,T)}~|| R_{g} ||^{n/2}_{L^{n/2}( B(x, r, t) )} \le \varepsilon,~~\forall x \in M,
\]
where $\epsilon$ denotes the positive constant as in the previous lemma.

Note that $(a')$, (b) implies (a) from the above non-collapsing estimate and volume upper bound.
Indeed, fix any $t \in [0,T)$ and let $N \in \mathbb{N}$ be the maximal number of collection of balls so that
$\{ B(x_{i}, r/2, t) \}_{1 \le i \le N}$ are disjoint each other and $\{ B(x_{i}, r, t) \}_{1 \le i \le N}$ are a covering of $M.$
Then, from the non-collapsing estimate and volume upper bound,
$N \le \frac{V_{2}}{\delta \cdot (r/2)^{n}},$
where $V_{2}, \delta,~r$ denote respectively the upper volume bound in Lemma \ref{volume}, the non-collapsing constant in Lemma \ref{non-collapsing}
and the radius in the definition of $(a').$
Moreover, by the definition of $N,$
\[
|| R_{g(t)} ||^{2}_{L^{2}(M)} \le \sum^{N}_{i = 1}~|| R_{g(t)} ||^{2}_{L^{2}(B(x, r, t))} \le \varepsilon N \le \varepsilon \cdot \frac{V_{2}}{\delta \cdot (r/2)^{n}}.
\]
\begin{rema}
By {\cite[Theorem~8.1]{bamler2020entropy}}, the pointed Nash entropy at $(x,t)$ is bounded by the volume growth of geodesic ball centered at $(x,t)$ from below. 
Hence, for closed Ricci flow $(M^{n}, g(t))_{t \in [0,T)}$ satisfying $(a'),$
we can adapt the theory in \cite{bamler2020structure}.
That is, take a sequence of points $(x_{i}, t_{i}) \in M \times [0,T)$ with $t_{i} \rightarrow T,$
then after application of a time-shift by $-t_{i},$
the flows $(M, g_{i}(s) := g(s + t_{i}))_{s \in [-t_{i}, 0]}$ satisfies the non-collapsedness assumption $(1.3)$ in \cite{bamler2020structure} as mentioned above (see {\cite[Part~1]{bamler2020structure}}).
\end{rema}

Next, we will show a distance distortion estimate.
\begin{lemm}[cf. {\cite[Theorem~1.1]{bamler2017heat}}]
\label{distance2}
Suppose that $(M^{n}, g(t))_{t \in [0,T)}$ satisfies $(a')$ and $(b).$
Then there exists a constant $\alpha = \alpha(g_{0}, n, T, C_{2})$
such that the following holds:
Suppose that $t_{0} \in [0, T ),~0 \le r_{0} \le \sqrt{t_{0}}.$
Let $x_{0}, y_{0} \in M$ with $d_{t_{0}} (x_{0}, y_{0}) \ge r_{0}$
and let $t \in [ t_{0} - \alpha r_{0}^{2}, \min \{ t_{0} + \alpha r_{0}^{2}, T \} ).$
Then
\[
\alpha d_{t_{0}} (x_{0}, y_{0}) < d_{t} (x_{0}, y_{0}) < \alpha^{-1} d_{t_{0}} (x_{0}, y_{0}).
\]
Here, $C_{2}$ denotes the upper bound in $(b)$ stated in section 1.
\end{lemm}
\begin{proof}
The proof of {\cite[Theorem~1.1]{bamler2017heat}} can be adapted to this case with some minor modifications.

\underline{Step~1}
:Assume that $r_{0} \le d_{t_{0}} (x_{0}, y_{0}) \le 2 r_{0}.$
We shall find the upper bound on $d_{t} (x_{0}, y_{0}).$
From {\cite[End~of~7.1]{perelman2002entropy}}
or {\cite[Chap~7]{morgan2007ricci}}
(see {\cite[Proof~of~Theorem~1.1]{bamler2017heat}}),
there exists a point $z \in M$ such that
\[
l_{(x_{0}, t_{0})} (z, t_{0} - \frac{1}{2} r_{0}^{2}) \le \frac{n}{2},
\]
where $l_{(x_{0}, t_{0})} (z, t_{0} - \frac{1}{2} r_{0}^{2})$
denotes the reduced distance between $(x_{0}, t_{0})$ and $(z, t_{0} - \frac{1}{2} r_{0}^{2}).$
Consider the heat kernel $K(x, t) = K(x, t; z, t_{0} - \frac{r_{0}^{2}}{2})$
centered at $(z, t_{0} - \frac{r_{0}^{2}}{2})$
(i.e., $\partial_{t} K = \Delta K$).
Then, from the scalar maximum principle, we obtain
\[
\begin{split}
\frac{d}{dt} \int_{M} K(\bullet, t) dvol_{t} 
&= \int_{M} \Delta K(\bullet, t) - R(\bullet, t) K(\bullet, t) dvol_{t} \\
&=\int_{M} - R(\bullet, t) K(\bullet, t) dvol_{t} \\
&\le \int_{M} - \min R(\bullet, t) K(\bullet, t) dvol_{t} \\
&\le \int_{M} - \min R(\bullet, 0) K(\bullet, t) dvol_{t}.
\end{split}
\]
Hence 
\[
\int_{M} K(\bullet, t) dvol_{t} \le e^{- \min R(g_{0}) (t - (t_{0} - \frac{r_{0}^{2}}{2}))}
\]
for all $t \in [t_{0} - \frac{r_{0}^{2}}{2}, t_{0} + \frac{r_{0}^{2}}{2}].$
From the heat kernel estimate (2.9), (2.10) in \cite{bamler2017heat},
there is some positive constant $B_{0},$ such that for all $t \in [t_{0} - \frac{r_{0}^{2}}{4}, t_{0} + \frac{r_{0}^{2}}{4}],$
\[
K(\bullet, t) < \frac{B_{0}}{\bigl( t - (t_{0} - r_{0}^{2} / 2) \bigr)^{n/2}}~~~~\mathrm{on}~M
\]
and
\[
K(x_{0}, t_{0}) \ge \frac{e^{- l_{(x_{0}, t_{0})} (z, t_{0} - r^{2}_{0} / 2)}}{\bigl( 4 \pi (t_{0} - (t_{0} - r^{2}_{0} / 2)) \bigr)^{n/2}}
\ge
\frac{e^{-n/2}}{(2 \pi r_{0}^{2})^{n/2}}.
\]
Hence we obtain
\[
K < B \cdot r_{0}^{- n}~~~~\mathrm{on}~M \times \left[t_{0} - \frac{r^{2}_{0}}{4}, t_{0} + \frac{r^{2}_{0}}{4} \right]
\]
for some positive constant $B.$
And, from the derivative estimate (2.8) in \cite{bamler2017heat},
\[
\frac{|\nabla K(\bullet, t')|}{K(\bullet, t')}
<
\sqrt{\frac{1}{t' - (t_{0} - r^{2}_{0} / 8)}} \sqrt{\log \frac{B \cdot r^{-n}_{0}}{K(\bullet, t')}}
\]
on $M$ for $t' \in [t_{0} - \frac{r^{2}_{0}}{8}, t_{0} + \frac{r^{2}_{0}}{8}].$
We can rewrite this as
\[
\left| \nabla \sqrt{\log \frac{B \cdot r^{-n}_{0}}{K(\bullet, t')}} \right| < \frac{1}{2 \sqrt{t' - (t_{0} - r^{2}_{0} / 8)}}.
\]
Let $\gamma : [0,1] \rightarrow M$ be a minimizing geodesic between $x_{0}$ and $y_{0}$ with respect to $g(t_{0}).$
Integrating the above bound at time $t_{0}$ along $\gamma$ and using the lower bound in the above yields that
\[
K(\bullet, t_{0}) > c r_{0}^{- n}~~~~\mathrm{on}~\gamma \left( [0,1] \right)
\]
for some constant $c > 0.$
Assume that $\alpha \in \min \left\{ 1/8, \alpha (B,C_{2},a) \right\}$
and fix some $t \in [t_{0} - \alpha r^{2}_{0}, t_{0} + \alpha r^{2}_{0}]$ 
then, by {\cite[Lemma~3.1~(a)]{bamler2017heat}} and the assumption (b),
for all $z \in \gamma ([0,1]),$
\[
\mathrm{sgn}(t_{0} - t) \left( |K(z, t_{0})| - |K(z, t)| \right) < \int^{t}_{t_{0}} \frac{B r^{-n}_{0}}{s} ds + (B r^{-n}_{0}) \int^{t}_{t_{0}} R (z, s)~ds
\]
and
\[
K(\bullet, t) > \frac{c r_{0}^{- n}}{2}~~~~\mathrm{on}~\gamma \left( [0,1] \right).
\]
We expalin this in more detail here.
From the assumption $(b)$ and the dominated convergence theorem, for any $p \in M,$
\[
[0, T] \ni t \mapsto \int^{T}_{t} R(p, s) ds
\] 
is continuous.
And, since $(g_{t})_{t \in [0,T)}$ is a smooth Ricci flow, 
\[
F : [0,T] \times M \rightarrow \mathbb{R}~;~(t, p) \mapsto \int^{T}_{t} R(p, s) ds
\]
is continuous.
Since $M$ is compact, $F$ is also uniformly continuous.
Hence there is a constant $\delta = \delta \left( \frac{c B^{-1}}{4} \right)$ such that
\[
\int^{t}_{t_{0}} R(z, s)~ds \le \frac{c B^{-1}}{4}~~~~\mathrm{if}~|t - t_{0}| < \delta.
\]

We go back to the proof.
From the above derivative estimate, for all $s \in [0,1],$
\[
K(\bullet, t) > \frac{c r_{0}^{- n}}{4}~~~~\mathrm{on}~B\left( \gamma(s), t, \beta r_{0} \right)
\]
for some uniform $\beta > 0.$
Then, there is a positive integer $N,$ $d_{t} (x_{0}, y_{0}) < N \cdot 4 \beta r_{0}$
and 
\[
e^{\max \{ 0, - \min R_{g_{0}} \} T}
\ge \int_{M} K(\bullet, t) dvol_{t}
> N \delta r_{0}^{n} \beta^{n} (c/4) r_{0}^{- n}.
\]
Here, $\delta$ denotes the positive constant as in the previous lemma.
Hnece 
\[
N < \frac{e^{\max \{ 0, - \min R_{g_{0}} \} T}}{\delta \beta^{n} (c/4)} =: N_{0}
\]
Then, from the asumption $r_{0} \le d_{t_{0}} (x_{0}, y_{0}) \le 2 r_{0},$
\[
d_{t} (x_{0}, y_{0}) < 4 N_{0} \beta r_{0}
\le 4 N_{0} \beta d_{t_{0}} (x_{0}, y_{0}).
\]
Therefore we have proven the upper bound for
\[
\alpha \le \min \left\{ \frac{1}{8}, \alpha(B, C_{2}, a), (4 N_{0} \beta)^{-1} \right\}
\]
in the case in which $r_{0} \le d_{t_{0}} (x_{0}, y_{0}) \le 2 r_{0}.$

\noindent
\underline{Step~2}
:The upper bound in the case in which $r_{0} \le d_{t_{0}} (x_{0}, y_{0}).$

This is proven in the same way as in Step 2 in the proof of {\cite[Theorem~1.1]{bamler2017heat}}.

\noindent
\underline{Step~3}
:The lower bound.

This is proven in the same way as in Step 3 in the proof of {\cite[Theorem~1.1]{bamler2017heat}}.
\end{proof}

Then, we have a corollary in the following:

\begin{coro}[cf. {\cite[Corollary~1.2]{bamler2017heat}}]
Assume a closed Ricci flow $(M^{n}, g(t))$ satisfies the same assumption in the previous lemma.
Then, there exists a constant $A = A (n, g(0), T, C_{2})$
the following holds:
Suppose that $t_{0} \in [0, T),~x_{0}, y_{0} \in M$ and $r_{0} := d_{t_{0}}(x_{0}, y_{0}).$
For any $t \in [r_{0}^{2}, T)$ with $A^{2}(r_{0}^{2} + |t - t_{0}|) \le t_{0},$
we have 
\[
d_{t} (x_{0}, y_{0}) < A \sqrt{r_{0}^{2} + |t - t_{0}|}.
\]
Here, $C_{2}$ denotes the upper bound in $(b)$ stated in section 1.
\end{coro}
\begin{proof}
This is proven in the same way as in the proof of Corollary 1.2 in \cite{bamler2017heat}
by using the previous lemma.
\end{proof}

Next, we prepare a key lemma to estimate the Ricci curvature pointwise by the Riemann scale.
To do this, we also need an additional assuumption $(a'')$ which is equivalent to the condition
$(*)_{p_{0}, \varepsilon}$ in our Main Theorems: Fix $p_{0} > n/2$ for dimension $n \ge 4.$
\[
(a'')~~~~~\mathrm{For~some}~r > 0,~\sup_{t \in [0,T)} || R_{t} ||^{n/2}_{L^{p_{0}}(B(x, t, r))} \le \frac{\tilde{\varepsilon}}{V_{2}^{1 - \frac{n}{2p_{0}}}},~\forall x \in M
\]
with $\tilde{\varepsilon} = \min \left\{  \varepsilon,~V_{2}^{1 - \frac{n}{2p_{0}}} \right\},$
where $\varepsilon,~V_{2}$ denote respectively the constant in the condition $(a')$ and 
the upper volume bound in Lemma \ref{volume}.
Note that $(a'')$ implies $(a')$ by H{\"o}lder's inequality since $p_{0} > n/2.$
Then we can obtain the useful estimates which is similar to one by Bamler and Zhang~\cite{bamler2017heat}.
In paticular, this is essential to show that the blow-up limits are Ricci flat (see Lemma \ref{limit} below):
\begin{lemm}[{\cite[Proposition~2.1]{jian2020global}}~(cf. {\cite[Lemma~6.1]{bamler2017heat}})]
\label{bamler6.1}
Assume $(a'')$ and $(b).$
Then there exist
$0 < \alpha < 1,~0< A_{0}, A_{1}, \cdots < + \infty$ which are universal constants depend on 
$n, g_{0}, T, C_{2}, p_{0}$ (where $C_{2}$ denotes the upper bound in $(b)$ stated in section 1)
such that the following holds :
Assume that $x_{0} \in M, t_{0} \in [0, T), 0 < r_{0}^{2} < \min \{ 1, t_{0} \}$
and
assume that $B(x_{0}, t_{0}, r_{0})$ is relatively compact and
$|\mathrm{Rm}| \le r_{0}^{-2}$ on $P(x_{0}, t_{0}, r_{0}, - r_{0}^{2}).$
Then the following holds :
\[
|\mathrm{Ric}|(x_{0}, t_{0}) < A_{0} r_{0}^{-1-\frac{n}{2p_{0}}},
\]
\[
|\partial_{t} \mathrm{Rm}|(x_{0}, t_{0}) < A_{0} r_{0}^{-3 - \frac{n}{2p_{0}}}.
\]
Here, $P(x_{0}, t_{0}, r_{0}, - r_{0}^{2})$ denotes the backward parabolic ball centered at $(x_{0}, t_{0})$
of radius $r_{0}$ defined as 
\[
P(x_{0}, t_{0}, r_{0}, - r_{0}^{2}) := B(x_{0}, t_{0}, r_{0}) \times ([t_{0}  -r_{0}^{2}, t_{0}] \cap [0,T)).
\]
Moreover, $|\nabla^{m} \mathrm{Ric}|(x_{0}, t_{0}) < A_{m} r_{0}^{-1-m -\frac{n}{2p_{0}}}$ for all $m \ge 0.$
\end{lemm}

\begin{rema}
Since $|\mathrm{Ric}|(x,t) \le \sqrt{n} (n-1) |\mathrm{Rm}|(x,t),$
we only have $|\mathrm{Ric}|(x_{0},t_{0}) \le \sqrt{n} (n-1) r_{0}^{-\mathbf{2}}.$
And, by Shi's estimate, we only have $|\partial_{t} \mathrm{Rm}|(x_{0}, r_{0}) < C r_{0}^{-\mathbf{4}}.$
\end{rema}

\begin{proof}[Proof of Lemma \ref{bamler6.1}]
The proof of {\cite[Proposition~2.1]{jian2020global}}
can be adapted to this case with minor modification.
We only indicate the necessary change. 
In the proof of {\cite[Proposition~2.1]{jian2020global}},
take $\alpha$ (the radius of the domain of the exponential map) so that $\alpha \le \min \{1, r \}$
(where $r > 0$ is the one in $(a'')$),
then we can argue in the same way in {\cite[Proposition~2.1]{jian2020global}}
and obtain desired assertions.
\end{proof}

\begin{defi}[Curvature radius~({\cite[Definition~1.7]{bamler2017heat}})]
For Riemannian manifold $(M, g)$ and a point $x \in M$ we define
\[
r_{|\mathrm{Rm}|} (x) := \sup \left\{ r > 0 \bigm| |\mathrm{Rm}| < r^{-2}~\mathrm{on}~B(x, r) \right\}.
\]
If $(M, g)$ is flat, then we set $r_{|\mathrm{Rm}|} = \infty.$
If $(M^{n}, g(t)_{t \in [0,T)})$ is a Ricci flow and
$(x, t) \in M \times [0,T),$
then $r_{|\mathrm{Rm}|} (x, t)$ is defined to the radius $r_{|\mathrm{Rm}|} (x)$ on the Riemannian manifold $(M, g(t)).$
\end{defi}

\begin{rema}
\label{curvrad}
The curvature radius $r_{|\mathrm{Rm}|} (x,t)$ is 1-Lipschitz in the variable $x$
(see {\cite[p. 768]{bamler2018convergence}}, see also {\cite[Theorem~2.2]{buzano2020local}}).
On the other hand, Buzano and Di Matteo \cite{buzano2020local} recently defined Riemann and Ricci scales of certain different notions. 
\end{rema}

Next, we will state the following backward pseudolocality theorem by Bamler
which will be needed in the proof of Main Theorem \ref{maintheo}:

\begin{lemm}[Backward Pseudolocality~{\cite[Theorem~1.48]{bamler2020structure}}]
\label{pseudo}
For any $n \in \mathbb{N}$ and $\alpha > 0$ there exists an $\varepsilon (n, \alpha) > 0$
such that the following holds:
Let $(M^{n}, g(t))_{t \in [0,T)}$ be a Ricci flow on a closed, $n$-dimensional smooth manifold and
let $(x_{0}, t_{0}) \in M \times [0,T)$ a point and $r >0$ a scale such that $[t_{0} - r^{2}, t_{0}] \subset [0,T)$
and
\[
\mathrm{Vol}_{t_{0}} \left( B(x_{0}, t_{0}, r) \right) \ge \alpha r^{n},~~~|\mathrm{Rm}| \le (\alpha r)^{-2}~~~\mathrm{on}~~~B(x_{0}, t_{0}, r).
\]
Then
\[
|\mathrm{Rm}| \le (\varepsilon r)^{-2}~\mathrm{on}~P \left( x_{0}, t_{0} ; (1 - \alpha) r, -(\varepsilon r)^{2} \right),
\]
where $P \left( x_{0}, t_{0} ; (1 - \alpha) r, -(\varepsilon r)^{2} \right)$ denotes the backward parabolic ball centered at $(x_{0}, t_{0})$
of radius $(1 - \alpha) r$ defined as in Lemma \ref{bamler6.1}.
\end{lemm}

We also need the following lemma to analyse the singularities:
\begin{lemm}[{\cite[Lemma~7.1]{bamler2017heat}}]
\label{picking}
Assume $(a')$ with $n \ge 4.$
Then there exists a constant $\delta > 0,$ which only depends on $n,~g(0)$ and $T$
such that the following holds:
Let $(x, t) \in M \times [0,T)$ and assume that $r_{|\mathrm{Rm}|}^{2} (x,t) \le t.$
Then there exists a point $y \in M$ such that 
\[
(1)~d_{t}(x, y) < 2 r_{|\mathrm{Rm}|}(x, t),
\]
\[
(2)~r_{|\mathrm{Rm}|}(y,t) \le r_{|\mathrm{Rm}|}(x,t),
\]
\[
(3)~\int_{B(y,t,\frac{1}{8} r_{|\mathrm{Rm}|}(y,t))} |\mathrm{Rm}|^{2} dvol_{t} > \delta.
\]
\end{lemm}

\begin{proof}
Since we have the non-collapsing estimate Lemma \ref{non-collapsing},
we can prove this assertion in exactly the same way as in {\cite[Proof~of~Lemma~7.1]{bamler2017heat}}.
\end{proof}

\begin{lemm}[$L^{2}$-curvature estimate in dimension 4~(cf.~{\cite[Theorem~1.8]{bamler2017heat}})]
\label{l-2}
Assume $(a'')$, (b) and $n = 4.$
Then there exists $A, B < + \infty,$
which depend on $C_{2}, p_{0}$ and $(M,g(0))$
such that the following holds
(Here, $C_{2}, p_{0}$ denote respectively the upper bound in $(b)$ stated in section 1 and the exponent of norm in the assumption $(a'')$):

\noindent
For all $t \in [T/2, T)$ the following bounds hold:
\[
|| \mathrm{Rm} (\cdot, t)||_{L^{2}(M, g(t))} \le A\,\chi(M) + B\,\mathrm{Vol}_{0} (M),
\]
where $\chi(M)$ denotes the Euler characteristic of $M.$

\noindent
For any $p \in (0, 4)$ and $t \in [0,T),$
\[
\int_{M} r_{|\mathrm{Rm}|}^{-p} (x,t) dvol_{t} \le A\,\chi(M) + \frac{B}{4-p} \mathrm{Vol}_{0}(M),
\] 
and for any $p \in (0, \frac{4 p_{0}}{2 + p_{0}}),$
\[
\int_{M} |\mathrm{Ric}|^{p}(x,t) dvol_{t} \le A\,\chi(M) + \frac{B}{4-p} \mathrm{Vol}_{0}(M).
\]

\noindent
Finally, for all $0 < s \le 1$ and $t \in [0,T),$
\[
\begin{split}
\frac{\mathrm{Vol}_{t} \left( \left\{ |\mathrm{Ric}|^{\frac{p_{0}}{2 + p_{0}}} (\cdot, t) \ge s^{-1} \right\} \right)}{s^{4}}
&+ \frac{\mathrm{Vol}_{t} \left( \left\{ r_{|\mathrm{Rm}|} (\cdot, t) \le s \right\} \right)}{s^{4}} \\
&\le A\,\chi(M) + B\,\mathrm{Vol}_{0}(M).
\end{split}
\]
\end{lemm}

\begin{rema}
Simon \cite{simon2015some} showed that
for closed four-dimensional Ricci flow with uniformly bounded  scalar curvature on a finite time interval $I$,
$\sup_{t \in I} ||\mathrm{Rm}||_{L^{2}(M)}$ and $|| \mathrm{Ric} ||_{L^{4}(M \times I)}$ are bounded.
This is proved by using the uniformly scalar curvature bound
and observating the evolution of the quantity
$\frac{|\mathrm{Ric}|^{2}}{R + 2}$
under the Ricci flow(see \cite{simon2015some} for more detail).
\end{rema}

\begin{proof}[Proof of Lemma \ref{l-2}]
This lemma can be proven in almost the same way as in {\cite[Proof~of~Theorem~1.8]{bamler2017heat}}.
To help understanding, we will describe these without omitting.

For the rest of this proof, fix some $t \in [T/2, T).$
Let $\delta > 0$ be the positive constant from Lemma \ref{picking} and set
\[
S := \left\{ y \in M \biggm| r_{|\mathrm{Rm}|}(y,t) < \bar{r}~\mathrm{and}~
\int_{B(y,t, \frac{1}{8} r_{|\mathrm{Rm}|})} |\mathrm{Rm}|^{2} dvol_{t} > \delta \right\}.
\]
Here, $\bar{r}$ is a fixed parameter determined later.
By the Vitali's covering theorem (see Lemma \ref{vitali} in Appendix of this paper), the upper volume bound(Lemma \ref{volume}) and the non-collapsing estimate(Lemma \ref{non-collapsing}),
there exist finitely many points $y_{1}, \cdots, y_{N} \in S$ such that for 
\[
r_{i} := r_{|\mathrm{Rm}|}(y_{i},t) < \bar{r},
\]
the balls $B \left( y, t, \frac{1}{8} r_{i} \right)$ are pairwise disjoint and 
\[
\bigcup^{N}_{i=1} B ( y_{i}, t, \frac{1}{2} r_{i} ) \supset \bigcup_{y \in S} B ( y, t, \frac{1}{8} r_{|\mathrm{Rm}|}(y, t) )
\supset S.
\] 

Then, for every $x \in M$ with $r_{|\mathrm{Rm}|}(x, t) \le \bar{r}$ there exists an $i \in \{ 1, \cdots, N \}$ for which
\[
d_{t}(x, y_{i}) < 2 r_{|\mathrm{Rm}|}(x, t) + \frac{1}{2} r_{i}.
\]
In fact, let  $x \in M$ with $r_{|\mathrm{Rm}|}(x, t) \le \bar{r}.$
By Lemma \ref{picking}, there exists a point $y \in M$ with $d_{t}(x,y) < 2 r_{|\mathrm{Rm}|}(x, t)$
and $r_{|\mathrm{Rm}|}(y, t) \le r_{|\mathrm{Rm}|}(x, t) < \bar{r}$ that satisfies the lower bound on the integral in the definition of $S.$ So $y \in S.$
By the above inclusion for covering of $S,$
there exists an index $i \in \{ 1, \cdots, N \}$ for which $y \in B \left( y_{i}, t, \frac{1}{2} r_{i} \right).$
Hence we conclude that
\[
d_{t}(x, y_{i}) \le d_{t}(x, y) + d_{t}(y, y_{i}) < 2 r_{|\mathrm{Rm}|} (x, t) + \frac{1}{2} r_{i}.
\]   

Next, observe that for any $p > 0,$ by Fubini's theorem,
\[
\begin{split}
\int_{M} r_{|\mathrm{Rm}|}^{-p} (x, t) dvol_{t} 
&= \bigintsss_{M} \left( \bar{r}^{-p} + \int_{r_{|\mathrm{Rm}|}(x, t)}^{\bar{r}} ps^{-p-1}~ds \right)~dvol_{t}(x) \\
&\le \bar{r}^{-p} \mathrm{Vol}_{t}(M) 
+ \int^{\bar{r}}_{0} p s^{-p-1}  \mathrm{Vol}_{t} \left( \left\{ r_{|\mathrm{Rm}|}(\cdot, t) \le s \right\} \right)ds.
\end{split}
\]

Set for $i = 1, \cdots, N,$
\[
\begin{split}
D_{i} = \bigl\{ x \in M \bigm| &r_{|\mathrm{Rm}|} (x, t) \le \bar{r}~\mathrm{and}~\\
&d_{t}(x, y_{i}) - \frac{1}{2} r_{i} \le d_{t}(x, y_{j}) - \frac{1}{2} r_{j}~\mathrm{for~all}~j = 1, \cdots, N \bigr\}.
\end{split}
\]
Then 
\[
\{ r_{|\mathrm{Rm}|}(\cdot,, t) \le \bar{r} \} = D_{1} \cup \cdots \cup D_{N}.
\]
[The inclusion $(\supset)$ is trivial.
The opposite inclusion $(\subset)$ is proved as follows:
For any $x \in \{ r_{|\mathrm{Rm}|}(\cdot,, t) \le \bar{r} \},$
take $y_{i}$ such that $d_{t}(x, y_{i}) - \frac{1}{2} r_{i} = \min_{1 \le j \le N} \left( d_{t}(x, y_{j}) - \frac{1}{2} r_{j} \right).$]
It follows from the above claim that for all $x \in D_{i},$
\[
r_{|\mathrm{Rm}|} (x, t) > \frac{1}{2} \left( d_{t}(x, y_{i}) - \frac{1}{2} r_{i} \right).
\]
Hence, if $d_{t}(x, y_{i}) \ge \frac{3}{4} r_{i},$ then $r_{|\mathrm{Rm}|} (x, t) > \frac{1}{2} \left(d_{t}(x, y_{i}) - \frac{1}{2} \cdot \frac{4}{3} d_{t}(x, y_{i}) \right) > \frac{1}{8} d_{t}(x, y_{i}).$
On the other hand, by the definition of the curvature radius, we obtain that if 
$d_{t}(x, y_{i}) < \frac{3}{4} r_{|\mathrm{Rm}|}(y_{i}, t) = \frac{3}{4} r_{i},$
then $r_{|\mathrm{Rm}|}(x,t) \ge \frac{1}{4}r_{i} > \frac{1}{3} d_{t}(x, y_{i}) > \frac{1}{8} d_{t}(x, y_{i})$
[Rem.: On $B(x, t, \frac{1}{4} r_{i}),~|\mathrm{Rm}| < r_{i}^{-2} < (\frac{1}{4} r_{i})^{-2},$
hence, by the definition of $r_{|\mathrm{Rm}|},~r_{|\mathrm{Rm}|}(x, t) \ge \frac{1}{4} r_{i}$ ].
Therefore, in either cases, 
\[
r_{|\mathrm{Rm}|}(x, t) > \frac{1}{8} d_{t}(x, y_{i})~~\mathrm{if}~~x \in D_{i}.
\]
Using this inequality, we can compute that for any $s \le \bar{r},$
\[
\begin{split}
\mathrm{Vol}_{t} (\{ r_{|\mathrm{Rm}|}(\cdot, t) \le s \})
&\le \sum^{N}_{i=1} \mathrm{Vol}_{t} (\{ r_{|\mathrm{Rm}|}(\cdot, t) \le s \} \cap D_{i}) \\
&\le \sum^{N}_{i=1} \mathrm{Vol}_{t} (\{ d_{t}(\cdot, y_{i}) \le 8 s \} \cap D_{i}) \\
&\le \sum^{N}_{i=1} \mathrm{Vol}_{t} (B(y_{i}, t, 8 s))
\le ( 8^{4} \cdot \kappa_{1} ) \cdot N s^{4}.
\end{split}
\]
Here, we used the non-inflating estimate(Lemma \ref{non-inflating})
and $\kappa_{1}$ denotes the positive constant in it.
Hence, using the above expression, for any $0 < p < 4,$
\[
\begin{split}
\int_{M} r_{|\mathrm{Rm}|}^{-p} (x, t) dvol_{t} 
&\le \bar{r}^{-p} \mathrm{Vol}_{t}(M) 
+ \int^{\bar{r}}_{0} p s^{-p-1}  \mathrm{Vol}_{t}( \{ r_{|\mathrm{Rm}|}(\cdot, t) \le s \} )~ds \\
&\le \bar{r}^{-p} \mathrm{Vol}_{t}(M) + C(\kappa_{1}) N p \int^{\bar{r}}_{0} s^{3-p} ds \\
&\le \bar{r}^{-p} \mathrm{Vol}_{t}(M) + C N \frac{p}{4-p} \bar{r}^{4-p}.
\end{split}
\]
Next, we assume that $\bar{r} < \min \{ 1, \sqrt{T/2} \}$ and observe that by the non-collapsing estimate(Lemma \ref{non-collapsing}),
the backward pseudolocality theorem
(Lemma \ref{pseudo})
and Lemma \ref{bamler6.1}, there exists a universal constant $\tilde{C}$ such that for all $x \in M$ for which $r_{|\mathrm{Rm}|}(x, t) \le \bar{r},$ we have
\[
|\mathrm{Ric}|(x, t) \le \tilde{C} r_{|\mathrm{Rm}|}^{-1 - \frac{2}{p_{0}}}(x, t).
\]
Hence by the above inequality for $p = 2 + \frac{4}{p_{0}},$
\[
\begin{split}
\int_{M} |\mathrm{Ric}(x, t)|^{2} dvol_{t} &\le \tilde{C}^{2} \int_{M} r_{|\mathrm{Rm}|}^{-2 - \frac{4}{p_{0}}}(x, t) dvol_{t} \\
&\le \tilde{C}^{2} \bar{r}^{-2 - \frac{4}{p_{0}}} \mathrm{Vol}_{t}(M)  + C \tilde{C}^{2} N \bar{r}^{2 - \frac{4}{p_{0}}}.
\end{split}
\]
On the other hand, recall that by the choice of the $y_{i} \in S$ we have that for all $i = 1, \cdots, N,$
\[
\int_{B(y, t, \frac{1}{8} r_{i})} |\mathrm{Rm}(x, t)|^{2} dvol_{t}(x) > \delta
\]
and the domains of these integrals are pairwise disjoint. Therefore
\[
N \delta < \int_{M} |\mathrm{Rm} (x, t)|^{2} dvol_{t}(x).
\]
Finally, we apply the Chern-Gauss-Bonnet theorem (in dimension 4):
\[
\int_{M} |\mathrm{Rm}(x, t)|^{2} dvol_{t}(x) = 32 \pi^{2} \chi(M) + \int_{M} (4 |\mathrm{Ric}(x, t)|^{2} - R^{2}(x, t))~dvol_{t}(x).
\]
Combinig this equality and the above inequalities gives us
\[
N \delta < 32 \pi^{2} \chi(M) + 4 \tilde{C}^{2} \bar{r}^{-2 -\frac{4}{p_{0}}} \mathrm{Vol}_{t}(M) + 4 C \tilde{C}^{2} N \bar{r}^{2 - \frac{4}{p_{0}}}.
\]
We now choose $\bar{r} > 0$ small enough such that
\[
4 C \tilde{C}^{2} \bar{r}^{2 - \frac{4}{p_{0}}} < \frac{1}{2} \delta.
\]
Then, by the proof of Lemma \ref{volume},
\[
\begin{split}
N < \frac{32 \pi^{2} \chi(M) + 4 \tilde{C}^{2} \bar{r}^{-2 - \frac{4}{p_{0}}} \mathrm{Vol}_{t}(M)}{\frac{1}{2} \delta}
&=: A\,\chi(M) + \Tilde{B}\,\mathrm{Vol}_{t}(M) \\
&\le A\,\chi(M) + B\,\mathrm{Vol}_{0}(M).
\end{split}
\]
The assertion follows from the above inequalities (see {\cite[proof~of~Theorem~1.8]{bamler2017heat}}).
\end{proof}

From the proof of the above $L^{2}$-estimate, we obtain the following.
\begin{lemm}[cf.~{\cite[Lemma~7.2]{bamler2017heat}}]
\label{points}
Assume $(a'')$, (b) and $n = 4$.
Then there exists a natural number $N \in \mathbb{N}$ such that for every time $t \in [0,T),$
we can find points $y_{1,t}, \cdots, y_{N,t} \in M$ such that for any $x \in M$ we have
\[
r_{|\mathrm{Rm}|}(x,t) > \frac{1}{8} \min_{i = 1, \cdots, N} d_{t} (x, y_{i,t}).
\]
\end{lemm}

\begin{proof}
This lemma can be proven in the same way as in {\cite[Lemma~7.2]{bamler2017heat}}.

Consider the points $y_{1}, \cdots, y_{N}$ and the radii $r_{i} = r_{|\mathrm{Rm}|}(y_{i}, t)$
as constructed in the proof of the above $L^{2}$-estimate
for $\bar{r} := \max_{M} r_{|\mathrm{Rm}|}(\cdot, t).$
Then, from its definition, we have $D_{1} \cup \cdots \cup D_{N} = M.$
So, from the proof of the $L^{2}$-estimate again,
we obtain the desired inequality.
And, from the $L^{2}$-estimate of the Riemannian curvature, we obtain that
$N$ is bounded by a constant which is independent of time.
\end{proof}

Here, we will make sure definitions of convergence.
\begin{defi}[{\cite[Definition~7.3]{bamler2017heat}}]
Let $(M, g(t))_{t \in [0,T)}~(T < + \infty)$ be a Ricci flow.
A pointed, complete metric space $(X, d, x_{\infty}), x_{\infty} \in X$ is called a limit of $(M, g(t))$
if there exists a sequence of times $t_{k} \rightarrow T,$ a sequence of points $x_{k} \in M$
and a sequence of numbers $\lambda_{k} \ge 1$ such that the sequence of length metric spaces $\{ (M, \lambda_{k}^{2}g_{t_{k}}, x_{k}) \}$
converge to $(X, d, x_{\infty})$ in the Gromov-Hausdirff sence.
If $\lim_{k \rightarrow \infty} \lambda_{k} = \infty,$
then $(X, d, x_{\infty})$ is called a blowup limit.
\end{defi}

From the non-inflating and non-collapsing estimates which are scaling invariant,
(blowup) limits always exists:
\begin{lemm}[Existence of a limit]
\label{existence}
Let $(M^{4}, g(t))_{t \in [0,T)}~(T < \infty)$ be a 4-dimensional closed Ricci flow satisfying $(a'')$ and (b).
Consider arbitrary sequence $t_{k} \rightarrow T,~x_{k} \in M$ and $\lambda_{k} \ge 1.$
Then, after passing to a subsequence, $(M, \lambda_{k}^{2} g_{t_{k}}, x_{k})$ converges to a complete limit
$(X, d, x_{\infty})$ in the Gromov-Hausdorff sense. 
\end{lemm}

\begin{proof}
From the non-collapsing estimates, the diameter upper bound and the volume upper bound, we can show that
every sequence taken from the flow
is uniformly totally bounded (see {\cite[Definition~7.4.13]{burago2001course}} for the definition of this notion).
Hence, using Theorem 7.4.15 in \cite{burago2001course}, we can show that every sequence has a limit.
From the non-inflating and non-collapsing estimates,
we can check that every rescaling sequence satisfies the assumption in Theorem 8.1.10 of \cite{burago2001course}.
Then we use this theorem and obtain the assertion for blowup limits.
\end{proof}

\begin{lemm}
\label{limit}
Let $(M^{4}, g(t))_{t \in [0,T)}~(T < \infty)$ be a 4-dimensional closed Ricci flow satisfying $(a'')$ and (b).
Let $(X, d, x_{\infty})$ be a limit of $(M, g(t)).$

\noindent
Then $\mathrm{diam}(X, d) > 0$ and there exist points $y_{1,\infty}, \cdots, y_{N,\infty} \in X$
such that the metric $d$ restricted to $X \setminus \{ y_{1,\infty}, \cdots, y_{N,\infty} \}$
is induced by a smooth Riemannian metric $g_{\infty}.$
Moreover, for all $x \in X,$
\[
|\mathrm{Rm}|(x) \le 64 \max_{i = 1, \cdots, N} d^{-2}(x, y_{i, \infty})
\]
and 
\[
|| \mathrm{Rm}_{g_{\infty}} ||_{L^{2}(X \setminus \{ y_{1,\infty}, \cdots, y_{N,\infty} \})} < +\infty.
\]
If $(X, d, x_{\infty})$ is even a blowup limit, then $\mathrm{Ric}_{g_{\infty}} \equiv 0.$
\end{lemm}

\begin{proof}
For any fixed scaling factor $\lambda \ge 1,$ consider the rescaling flow $(\lambda \cdot g(\lambda^{-1} \cdot)).$
Then the non-inflating estimate (Lemma \ref{non-inflating}) also holds for the same constant $\kappa_{1}.$
Moreover, in the proof of Lemma \ref{volume},
\[
\log \frac{\mathrm{Vol} (M, \lambda g_{t})}{\mathrm{Vol} (M, \lambda g_{0})} 
= \log \frac{\mathrm{Vol} (M, g_{t})}{\mathrm{Vol} (M, g_{0})} \ge - C_{2}.
\]
Hence we have the same volume-lower-bound as in Lemma \ref{volume} for the same constant $V_{1} > 0.$
Therefore we obtain the diameter lower bound (from the proof of the diameter lower bound as above).

Consider the points $y_{1, t_{k}}, \cdots, y_{N, t_{k}}$ from Lemma \ref{picking}.
After passing to a subsequence, we may assume that these points converge to points $y_{1,\infty}, \cdots, y_{N,\infty} \in X.$
By the backward pseudolocality(Lemma \ref{pseudo}) and Shi's estimates, we have uniform bounds on the covariant derivatives of the curvature tensor at uniform distance away from $y_{i, t_{k}}.$
Hence $(M, g_{t_{k}})$ smoothly converges to a Riemannian metric $g_{\infty}$ on $X \setminus \{ y_{1,\infty}, \cdots, y_{N,\infty} \}.$
The curvature bounds decend to the limit.
The vanishing of $\mathrm{Ric}_{g_{\infty}}$ in the blowup case is a consequence of Lemma \ref{bamler6.1}.
\end{proof}

\begin{lemm}
\label{tangent}
Let $(M^{4} , g(t))_{t \in [0,T)}$ and $(X, d, x_{\infty})$ be ones as in the previous lemma.
Then the tangent cone at every point of $(X, d)$ is isometric to a finite quotient of Euclidean space $\mathbb{R}^{4}.$
So $(X, d)$ is diffeomorphic to an orbifold with cone singularities.
\end{lemm}

\begin{proof}
The proof is similar to that of {\cite[Lemma~7.6]{bamler2017heat}}.

It follows from the second estimate in the previous lemma,
the non-inflating estimate(Lemma \ref{non-inflating})
and the non-collapsing estimate(Lemma \ref{non-collapsing}) 
that the tangent cones around any $y_{\infty}$
are flat away from the tip and unique (see {\cite[Section~8]{simon2015extending}}, {\cite[Section~3]{tian1990calabi}}).
It follows that the cross-section of each of these tangent cones is connected in the almost same way as in {\cite[Proof~of~Lemma~7.6]{bamler2017heat}} with slight modification.
We only indicate the necessary changes.
To lead to a contradiction for the proof of this fact, 
it is sufficient to show that if $(\mathbb{R} \times N^{3}, dt^{2} + g_{N})$
is a Ricci flat, then the whole space is flat.
But, since $N$ is theree dimensional, $(N, g_{N})$ is flat.
Hence $(\mathbb{R} \times N^{3}, dt^{2} + g_{N})$ is flat.
\end{proof}

We can prove Main Theorem \ref{maintheo}.
\begin{proof}[Proof of Main Theorem~\ref{maintheo}]
The proof is similar to that of {\cite[Corollary~1.11]{bamler2017heat}}.
We define the function $r_{|\mathrm{Rm}|}^{\infty} : M \rightarrow [0,\infty]$ as follows:
\[
r_{|\mathrm{Rm}|}^{\infty} (x) := \limsup_{t \rightarrow T} r_{|\mathrm{Rm}|} (x,t). 
\]
Consider a point $x \in M$ with $r := r_{|\mathrm{Rm}|}^{\infty}(x) > 0.$
So there exists a sequence of times $t_{k} \rightarrow T$
such that $r_{|\mathrm{Rm}|}(x, t_{k}) > r/2.$
Using the backward pseudolocality theorem (Lemma \ref{pseudo}), we find that 
\[
|\mathrm{Rm}| < 4 K r^{-2}~~\mathrm{on}~~\bigcup^{\infty}_{k = 1} P(x, t_{k}, \frac{1}{2} \varepsilon r, -(\frac{1}{2} \varepsilon r)^{2}).
\]
By the distance distortion estimate (Lemma \ref{distance2}), there exists a small $\tau > 0,$ and universal constant $C > 0$ such that
\[
P(x, T-\tau, \frac{1}{2C} \varepsilon r, \tau) 
\subset \bigcup^{\infty}_{k=1} P(x, t_{k}, \frac{1}{2}\varepsilon r, -(\frac{1}{2}\varepsilon r)^{2}),
\]
where $P(x, T-\tau, \frac{1}{2C} \varepsilon r, \tau)$ denotes the forward parabolic ball centered at $(x, T - \tau)$
of radius $\frac{1}{2C} \varepsilon r$ defined as
\[
P \left(x, T-\tau, \frac{1}{2C} \varepsilon r, \tau \right) := B \left(x, T - \tau, \frac{1}{2C} \varepsilon r \right) \times \left([T - \tau, T] \cap [0,T) \right).
\]
Hence as $t \rightarrow T,$ the metric $g(t)$ converges smoothly to a Riemannian metric $g_{T}$ in a neighborhood
of $x.$
Moreover, from this inclusion,
it follows that $r^{\infty}_{|\mathrm{Rm}|}$ is positive in this small neighborhood of $x.$
Hence we obtain that $M^{\mathrm{reg}} := \{ r^{\infty}_{|\mathrm{Rm}|} > 0 \} \subset M$ is an open subset and 
that $g(t)$ smoothly converges to a Riemannian metric $g_{T}$ on $M^{\mathrm{reg}}$ as $t \rightarrow T.$
Moreover, $r^{\infty}_{|\mathrm{Rm}|}$ restricted to $M^{\mathrm{reg}}$ is 1-Lipschitz with respect to $g_{T}$
(being the limsup of a family of 1-Lipschitz functions (see the above Remark \ref{curvrad})).
So by this 1-Lipschitz property and continuity of $g(t)$ (on $M^{\mathrm{reg}}$),
for any $x \in M^{\mathrm{reg}}$ and $\delta > 0$ the ball $B(x, t, r_{|\mathrm{Rm}|}(x) - \delta)$
is contained in $M^{\mathrm{reg}}$ for $t$ sufficiently close to $T.$
We also have $|\mathrm{Rm}|(\cdot, T) \le (r_{|\mathrm{Rm}|}^{\infty}(x))^{-2}$ on $B(x, t, r_{|\mathrm{Rm}|}(x) - \delta)$
for each $\delta > 0.$ Hnece we obtain
\[
r^{\infty}_{|\mathrm{Rm}|}(x) = \lim_{t \rightarrow T} r_{|\mathrm{Rm}|}(x,t)~~~~\mathrm{for~all}~x \in M.
\] 
Next, we will show that $M^{\mathrm{sing}} := M \setminus M^{\mathrm{reg}} = \{ r^{\infty}_{|\mathrm{Rm}|} = 0 \}$
is a zero set with respect to the Riemannian volume measure $dvol_{g(t)}$ for any $t \in [0,T).$
Since $\lim_{t \rightarrow T} r_{|\mathrm{Rm}|} (\cdot, t) = 0$ on $M^{\mathrm{sing}},$
we have with Lemma \ref{l-2} that for any $0 < s < 1,$
\[
\begin{split}
\mathrm{Vol}_{t} (M^{\mathrm{sing}})
&\le \mathrm{Vol}_{t} (\{ r^{\infty}_{|\mathrm{Rm}|} < s \}) \\
&= \mathrm{Vol}_{t} (\{ \limsup_{\tilde{t} \rightarrow T} r_{|\mathrm{Rm}|}(\cdot, \tilde{t}) < s \}) \\
&= \limsup_{\tilde{t} \rightarrow T} \mathrm{Vol}_{t} (\{ r_{|\mathrm{Rm}|}(\cdot, \tilde{t}) < s \}) \\
&\le e^{C_{2}} \limsup_{\tilde{t} \rightarrow T} \mathrm{Vol}_{\tilde{t}} (\{ r_{|\mathrm{Rm}|}(\cdot, \tilde{t}) < s \}) \\
&\le e^{C_{2}} C s^{4},
\end{split}
\]
for some positive constant $C = C(M, g(0), C_{1}, C_{2}, n, T).$
Letting $s \rightarrow 0$ yields $\mathrm{Vol}_{t} (M^{\mathrm{sing}}) = 0.$

Consider the points $y_{1,t}, \cdots, y_{N,t}$ from Lemma \ref{points}.
We claim that for any $0 < s < 1$ and any $t < T$ sufficiently close to $T,$ we have 
\[
M^{\mathrm{sing}} \subset \bigcup^{N}_{i = 1} B(y_{i, t}, t, s).
\]
Fix $s$ and choose $\tilde{s} > 0$ such that $\tilde{s} < \frac{1}{40} s$ and $\mathrm{Vol}_{t} (\{ r^{\infty}_{|\mathrm{Rm}|} < \tilde{s} \}) < \kappa_{2} (\frac{1}{2} s)^{4}$ for any $t \in [0,T)$
(this is possible because $M^{\mathrm{sing}}$ is zero set with respect to $dvol_{g(t)}$ and volume distortion).
Here, $\kappa_{2}$ denotes the positive constant in the non-collapsing estimate (Lemma \ref{non-collapsing}).
Define $U := \{ r^{\infty}_{|\mathrm{Rm}|} \ge \tilde{s} \} \subset M$
and choose $\tau > 0$ small enough such that $r_{|\mathrm{Rm}|} > \frac{1}{2} \tilde{s}$
on $U \times [T - \tau, T)$ and $r_{|\mathrm{Rm}|} < 2 \tilde{s}$ on $\partial U \times [T - \tau, T).$
Let $t \in [T - \tau, T)$ and $x \in M^{\mathrm{sing}} \subset M \setminus U(= \{ r^{\infty}_{|\mathrm{Rm}|} < \tilde{s} \}).$
We claim that there exists a point $z \in \partial U$ such that $d_{t}(z, x) < \frac{1}{2} s.$
Otherwise $B(x,t, \frac{1}{2} s) \cap U = \emptyset,$ which would imply that
\[
\kappa_{2} (\frac{1}{2} s)^{4} < \kappa_{2} s^{4} \le \mathrm{Vol}_{t} (B(x,t,s))
\le \mathrm{Vol}_{t} (M \setminus U)
< \kappa_{2}(\frac{1}{2} s)^{4}
\] 
which is a contradiction.
Since $r_{|\mathrm{Rm}|}(z,t) < 2 \tilde{s} < \frac{1}{20} s,$
there exists an $i \in \{ 1, \cdots, N \}$
such that $d_{t}(z, y_{i,t}) < 8 \cdot \frac{1}{20} s < \frac{1}{2} s$
by Lemma \ref{points}.
It follows that $x \in B(y_{i,t}, t, s).$
This proves the above inclusion.
This inclusion implies that $(M, g(t))$ converges to a metric space $(X, d)$ in the Gromov-Hausdorff sense
which is isometric to $(M^{\mathrm{reg}}, g_{T})$ away from at most $N$ points.
By the above Lemma \ref{tangent},
the tangent cone of $(X, d)$ are finite quotients of $\mathbb{R}^{4}.$
It follows that $M^{\mathrm{reg}}$ is connected and that $(X, d)$ isometric to the metric completion
$(\Bar{M}^{\mathrm{reg}}, d_{g_{T}})$ of $(M^{\mathrm{reg}}, g_{T}).$
Therefore $\Bar{M}^{\mathrm{reg}}$ can be endowed with a smooth orbifold structure such that $g_{T}$
is a smooth metric away from the singular points. This shows assertions (1)-(4).
The assertion (5) is also proved in the same way as in {\cite[proof~of~Corollary~1.11(e)]{bamler2017heat}}
after slight modifications. So we shall only indicate the necessary changes.
Under our assumptions, (7.19) and (7.20) in {\cite[proof~of~Corollary~1.11(e)]{bamler2017heat}} hold
with 
\[
\delta_{i} := C (2^{-i})^{2} || \mathrm{Rm} ||_{L^{\infty} (A_{i-1} \cup A_{i} \cup A_{i+1} ; g_{T})} + C (2^{-i})^{1 - \frac{n}{2p_{0}}}
\]
(where $n = 4$ in this case).
Moreover, as in the same way in {\cite[proof~of~Corollary~1.11(e)]{bamler2017heat}}, we have
\[
\delta_{i}^{2} \le C || \mathrm{Rm} ||^{2}_{L^{2}(A_{i-2} \cup \cdots \cup A_{i+2} ; g_{T})} + C (2^{-i})^{ 2 - \frac{n}{2p_{0}}}
\] 
(where $n = 4$ in this case). Hence $\sum^{\infty}_{i = 1} \delta^{2}_{i} < \infty.$
Therefore we can show (5) by the same arguments in {\cite[proof~of~Corollary~1.11(e)]{bamler2017heat}}
with $\delta_{i}$ being replaced with the above one.
\end{proof}

\section{Under Riemannian curvature integral bounds}
Next, we will consider closed Ricci flow satisfying the conditions $(A), (a'')$ and $(b)$ stated in section 1.
And prove the second main result ($=$ Main Theorem \ref{maintheo2}).
\begin{proof}[Proof of Main Theorem \ref{maintheo2}]
Firstly, we have a volume upper bound from the proof of Lemma \ref{volume}.
And, from (b), we can obtain a volume lower bound.
Since (A) implies (a), we can obtain diameter bounds and non-inflating estimate in the same way.
By the volume upper bound and H{\"o}lder's inequality, $(a'')$ implies $(a').$
Hence the non-collapsing estimate holds by Lemma \ref{non-collapsing}.
Moreover, 
if we assume the condition $(a''),$ we can also obtain exactly the same assertion in Lemma \ref{bamler6.1}.
We can also obtain the lemma corresponding to Lemma \ref{picking} but
we have to replace (3) in Lemma \ref{picking} with the following $L^{n/2}$-condition
\[
\int_{B(y,t,\frac{1}{8} r_{|\mathrm{Rm}|}(y,t))} |\mathrm{Rm}|^{n/2} dvol_{t} > \delta.
\]
We already have $L^{n/2}$-estimate from the condition (A),
but the proof of Lemma \ref{l-2} is still useful.
From the proof of Lemma \ref{l-2}, we can obtain the lemma corresponding to Lemma \ref{points}.
From non-inflating, non-collapsing estimates and the above corresponding lemmas,
we obtain the lemmas corresponding to Lemma \ref{existence} and \ref{limit}.
But we have to replace the second estimate in Lemma \ref{limit} with the $L^{n/2}$-estimate
\[
|| \mathrm{Rm}_{g_{\infty}} ||_{L^{n/2}(X \setminus \{ y_{1,\infty}, \cdots, y_{N,\infty} \})} \le C < +\infty.
\]
From this lemma corresponding to Lemma \ref{limit}, non-inflating estimate and non-collapsing estimate,
we obtain that the tangent cone at every point of the limit is isometric to a finite quotient of $\mathbb{R}^{n}$
and the limit is diffeomorphic to an orbifold with cone singularities (in the same way in {\cite[Section~3]{tian1990calabi}}).
To lead to a contradiction for the proof of the claim corresponding to Lemma \ref{tangent}, it is sufficient to show that if $(\mathbb{R} \times N^{n-1}, dt^{2} + g_{N})$
is a Ricci flat $n$-dimensional($n \ge 4$) Riemannian manifold, then
the space is flat.
Moreover, by Corollary \ref{ale} below, we also assume that $(\mathbb{R} \times N, dt^{2} + g_{N})$ is an ALE space (see \cite{bando1989construction} for definition of ALE).
If $n$ is odd, then the space is isometric to $(\mathbb{R}^{n}, g_{\mathrm{eucl.}})$ from {\cite[Theorem~(1.5)]{bando1989construction}}(see {\cite[Remarks.~2)]{bando1989construction}}).
If $n$ is even, then it follows that  $(N, g_{N})$ is isometric to $(\mathbb{R}^{n-1}, g_{\mathrm{euci.}})$ as in the case that $n$ is odd.
Hence the space $(\mathbb{R} \times N^{n-1}, dt^{2} + g_{N})$ is flat in the both cases.
Therefore we can obtain Main Theorem \ref{maintheo2}
from the above corresponding lemmas and backward pseudolocality(Lemma \ref{pseudo})
in the same way as in the proof of Main Theorem \ref{maintheo}.
\end{proof}

Since the $L^{2}$-bound of $\mathrm{Rm}$ is scaling invariant and descends to geometric limits,
we immediately obtain the following corollary from {\cite[Theorem~1.5]{bando1989construction}}:
\begin{coro}[cf. {\cite[Corollary~1.9]{bamler2017heat}}, \cite{bando1989construction}]
\label{ale}
Let $(M^{n}, g(t))_{t \in [0,T)}~(T < +\infty)$ be a Ricci flow on a closed $n$-manifold.
Suppose that $n = 4$ and $(M^{4}, g(t))_{t \in [0,T)}$ satisfies $(a'')$, (b),
or
$n \ge 4$ and $(M^{n}, g(t))_{t \in [0,T)}$ satisfies $(A'), (a'')$ and $(b).$
Let $(x_{k}, t_{k}) \in M \times [0,T)$ be a sequence with $Q_{k} = |\mathrm{Rm}|(x_{k}, t_{k}) \rightarrow \infty$
and suppose that the pointed sequence of rescaling flows $(M, Q_{k} g_{Q^{-1}_{k} t + t_{k}}, x_{k})$
converges to some ancient Ricci flow $(M_{\infty}, (g_{\infty,t})_{t \in (-\infty, 0]}, x_{k})$
in the smooth Cheeger-Gromov-Hamilton sense.
Then $g_{\infty, t} = g_{\infty}$ is constant in time and $(M_{\infty}, g_{\infty})$
is Ricci-flat and asymptotically locally Euclidean. 
\end{coro}

A direct consequence of the previous corollary is the following (cf. {\cite[Proposition~5.6]{anderson2008survey}}, {\cite[Corollary~1.10]{bamler2017heat}} ):
\begin{coro}[cf. {\cite[Corollary~1.10]{bamler2017heat}}]
Let $(M^{4}, g(t))_{t \in [0,T)}~(T < +\infty)$ be a Ricci flow on a closed 4-manifold $M$
satisfying the following topological condition: 
the second homology group over every field vanishes, i.e., $H_{2}(M ; \mathbb{F}) = 0$
for every field $\mathbb{F}.$
Then at least one of conditions $(a'')$, (b) fails.
\end{coro}

\section{Extend the flow}
\label{extend}
Next, we will state that the Ricci flow as above can be extended over $T$ as an orbifold Ricci flow.
Let $(X, d)$ be a limit of a closed Ricci flow $(M^{n}, g(t))_{t \in [0,T)}$ satisfying the assumptions
of Main Theorem \ref{maintheo} or \ref{maintheo2}.
Then, as we seen above, $(X, d)$ is a $C^{0}$ Riemannian orbifold.
This can be extended as an orbifold Ricci flow in the following sense :
\begin{theo}[Extendable as an orbifold Ricci flow~({\cite[Theorem~9.1]{simon2015extending}})]
There exists a smooth solution to the orbifold Ricci flow $(X, h(t))_{t \in (0,S)}$
for some $S > 0$ such that $(X, d (h(t))) \rightarrow (X, d)$
in the Gromov-Hausdorff sense as $t \rightarrow 0.$
\end{theo}

\begin{proof}
We can prove this claim by using the argument in {\cite[Section~9]{simon2015extending}}.
Firstly, we can flow each $(X_{i}, \bar{g}_{i^{-1}})$ by the orbifold Ricci flow $(X, Z_{i}(t)_{t \in (0,S)})$
(where $S$ is independent of $i$)
satisfying higher order derivative (Shi type) estimates.
Here, $\bar{g}_{i^{-1}}$ denotes the smooth orbifold metric constructed in (5) in Main Theorem
\ref{maintheo} or \ref{maintheo2} with $\varepsilon = i^{-1}.$
Then, using this derivative estimates and the compactness theorem for orbifold Ricci flows,
as $i \rightarrow \infty,$ we obtain a smooth orbifold Ricci flow $(X, h(t)_{t \in (0,S)})$
such that
\[
d_{\mathrm{GH}} ((X,d), (X, d_{h(t)})) \rightarrow 0
\]
as $t \rightarrow 0.$
This $(X, h(t)_{t \in (0,S)})$ is desired orbifold Ricci flow.
See {\cite[Section~9]{simon2015extending}} for more detail.
\end{proof}

\begin{rema}
Under the assumptions $(a'')$ and $(c),$ if we assume additionally that the volume or the diameter of $(M, g(t))_{t \in [0,T)}$ are bounded from below away from zero,
then the same result as in Main Theorem \ref{maintheo} holds.
The proof is as follows. :
The non-inflating estimate comes from (c) by Bishop-Gromov volume comparison.
Moreover, from (c), we have
\[
(*)~~B(x, r, s) \subset B(x, e^{\frac{C_{3}(t - s)}{2}}, t),~~s \le t,~r > 0,~x \in M.
\]
Hence, we use $(*)$ instead of distance distortion in the proof of Main Theorem \ref{maintheo}
and obtain the claim in the same way as in the proof of Main Theorem \ref{maintheo}.
\end{rema}

As we stated in the introduction,
in dimension 2 and 3,
the conditions (a) and (b) are sufficient to extend the flow smoothly over $T.$
From Main Theorem \ref{maintheo}, in dimension 4, we hope that the conditions (a) and (b) are sufficient to extend the flow over $T$
in the sense of Theorem \ref{extend}.

\section{Appendix}
To be self-contained, we will describe the Vitali's covering lemma used in the proof of Lemma \ref{l-2} :
\begin{lemm}[Vitali's covering lemma]
\label{vitali}
Let $\mathcal{F}$ be a collection of (nondegenerate) balls in a metric space, with bounded radii.
Then there exists a disjoint subcollection $\mathcal{G}$ of $\mathcal{F}$ with the following property:
every ball $B$ in $\mathcal{F}$ intersects a ball $C$ in $\mathcal{G}$ such that $B \subset 4 C,$
where $4 C$ denotes the ball with the same center of $C$ and four times the radius of $C.$
\end{lemm}

\begin{proof}
Let $R := \sup \{ \mathrm{rad} (B) | B \in \mathcal{F} \}$ and define
\[
\mathcal{F}_{n} := \left\{ B \in \mathcal{F} \bigm| \mathrm{rad} (B) \in ( (3/2 )^{-n-1} R,~(3/2 )^{-n} R~] \right\}~(0 \le n \in \mathbb{N}).
\]
Here, $\mathrm{rad}(B)$ denotes the radius of $B$ in the metric space.
$\mathcal{G}_{n} \subset \mathcal{F}_{n}$ and $\mathcal{H}_{n} \subset \mathcal{F}_{n}$ are defined inductively as follows:

\noindent
$\mathcal{H}_{0} := \mathcal{F}_{0}$ and $\mathcal{G}_{0} := \{ \mathrm{the~maximal~disjoint~subcollection~of}~ \mathcal{H}_{0} \}.$

\noindent
Assuming $\mathcal{G}_{0}, \mathcal{G}_{1}, \cdots, \mathcal{G}_{n}$ have been defined, let
\[
\mathcal{H}_{n+1} := \left\{ B \in \mathcal{F}_{n+1} \biggm| B \cap C = \emptyset,~\forall C \in \mathcal{G}_{0} \bigcup \cdots \bigcup \mathcal{G}_{n} \right\}
\]
and $\mathcal{G}_{n+1}$ be the maximal disjoint subcollection of $\mathcal{H}_{n+1}.$
Then
\[
\mathcal{G} := \bigcup^{\infty}_{n=0} \mathcal{G}_{n} \subset \mathcal{F}
\]
satisfies the requirements.
Indeed, let $n$ be such that $B$ belongs to $\mathcal{F}_{n}.$
Either $B$ does not belong to $\mathcal{H}_{n},$ which implies $n > 0$ and means that $B$ intersencts a ball from the
union of $\mathcal{G}_{0}, \cdots, \mathcal{G}_{n-1},$ 
or $B \in \mathcal{H}_{n}$ and by maximality of $\mathcal{G}_{n},$ $B$ intersects a ball in $\mathcal{G}_{n}.$
In any case, $B$ intersect a ball $C$ that belongs to the union of $\mathcal{G}_{0}, \cdots, \mathcal{G}_{n}.$
Such a ball $C$ has radius $> (\frac{3}{2})^{-n-1} R.$
Since the radius of $B$ is $\le (\frac{3}{2})^{-n} R,$
we obtain that $B \subset (1 + 2 \cdot (\frac{3}{2})) C = 4 \cdot C.$
\end{proof}

\subsection*{Acknowledgement}
~~I would like to thank my supervisor Kazuo Akutagawa for helpful conversations and support.
I would also like to thank the referee for pointing out the lack of explanation and typos in the first draft of this paper.

\bigskip
\textit{E-mail adress}:~a19.fg4w@g.chuo-u.ac.jp

\textsc{Department Of Mathematics, Chuo University, Tokyo 112-8551, Japan}

\end{document}